\documentclass[12pt,letterpaper]{article}

\usepackage{renstyles}

\usepackage[numbers,sort&compress]{natbib} 
\numberwithin{equation}{section}

\title{Unique determination of cost functions in a multipopulation mean field game model}
\author{
Kui Ren\thanks{Department of Applied Physics and Applied Mathematics, Columbia University, New York, NY 10027; kr2002@columbia.edu}
\and Nathan Soedjak\thanks{Department of Applied Physics and Applied Mathematics, Columbia University, New York, NY 10027; ns3572@columbia.edu}
\and
Kewei Wang\thanks{School of Mathematical Sciences, Peking University, Beijing 100871, China; wangkw@stu.pku.edu.cn}
}

\begin{document}

\maketitle



\begin{abstract}

This paper studies an inverse problem for a multipopulation mean field game (MFG) system where the objective is to reconstruct the running and terminal cost functions of the system that couples the dynamics of different populations. We derive uniqueness results for the inverse problem with different types of available data. In particular, we show that it is possible to uniquely reconstruct some simplified forms of the cost functions from data measured only on a single population component under mild additional assumptions on the coupling mechanism. The proofs are based on the standard multilinearization technique that allows us to reduce the inverse problems into simplified forms.

\end{abstract}





\section{Introduction}
\label{SEC:Intro}

This paper studies inverse problems to a multipopulation mean field game (MFG) model. Let $x\in\bbR^d$ denote the state variable, $t\in[0,\infty)$ denote the time variable, and $\bbT^d:=\bbR^d/\bbZ^d$ be the standard $d$-dimensional torus. The $n$-population mean field game model we are interested in takes the form, 
\begin{equation}\label{EQ:MFG MultiPop}
    \begin{array}{>{\displaystyle}lc}
        -\partial_t u_i(x,t)-\Delta u_i(x,t)+\frac{1}{2}|\nabla u_i(x,t)|^2=F_i(x,\bm(x,t)), &(x,t)\in\bbT^d\times(0,T)\\[1ex]
        \partial_t m_i(x,t)-\Delta m_i(x,t)-{\rm div}(m_i(x,t)\nabla u_i)=0, &(x,t)\in\bbT^d\times(0,T)\\[1ex]
        u_i(x,T)=G_i(x,\bm(x,T)),\ m_i(x,0)=m_{i,0}(x), &x\in \bbT^d
    \end{array}
\end{equation}
$i=1,2,\dots,n$, where $(u_i,m_i)$ is the value function-density pair of the $i$-th population, and $\Delta$, $\rm div$ are respectively the Laplacian and divergence operators with respect to $x$-variable. For simplicity we denote $\bu=(u_1,\dots,u_n)$, $\bm=(m_1,\dots,m_n)$. 
The functions $F_i$ and $G_i$ are, respectively, the running cost and terminal cost of the $i$-th population. The initial density of the $i$-th population is denoted by $m_{i,0}$.
 
Each component of the multipopulation MFG system \eqref{EQ:MFG MultiPop} consists of two coupled parabolic equations, which can be understood as nonlinear optimal transport equations. Players' value function $u_i$ satisfies the Hamilton-Jacobi-Bellman (HJB) equation, which is the first backward parabolic equation in \eqref{EQ:MFG MultiPop}, that describes the evolution of the cost of the players due to their choice of strategies. The density of each population $m_i$ satisfies the standard Fokker-Planck-Komogorov (FPK) equation, which is the second forward equation in \eqref{EQ:MFG MultiPop}, that describes the dynamics of population densities. These two coupling equations fully characterize the dynamics of the population. The coupling between different populations is mainly through the running cost functions $F_i$ and the terminal cost functions $G_i$ as they depend on the densities of all populations.

Mean field game models were originally introduced as the limit of a large class of multi-player games as the number of players increases toward infinity; see~\cite{HuCaMa-IEEE03,HuCaMa-JSSC07,HuCaMa-IEEE07,CaHuMa-CIS06,LaLi-CRM06I,LaLi-CRM06II,LaLi-JJM07} for some of the earliest works, and~\cite{AcLa-MFG20,Caines-ESC21,GoPiVo-Book16,GoSa-DGA14,BeFrYa-Book13,CaPo-MFG20,Ryzhik-Notes18,UlSwGo-PR19} and references therein for overviews of recent progress in the field. The Nash equilibrium of such games can be characterized as solutions of the MFG system. In a standard single-population MFG model, players are assumed to be all the same. However, in real-world modeling, one often needs to take into account interactions between players at different scales. This motivated the generalization to multipopulation MFG models such as ~\eqref{EQ:MFG MultiPop}; see~\cite{ArMa-SIAM22,BaCi-PDEMMAP18,BaFe-NoDEA16,Feleqi-DGA13,Gueant-AMO15,Wikecek-DGA24} and references therein.


In this work, we are interested in reconstructing the cost functions 
\[
    \bF:=(F_1, \dots, F_n),\quad\mbox{and},\quad \bG:=(G_1, \dots, G_n)\,.
\]
of the MFG system~\eqref{EQ:MFG MultiPop} from observed data on the value function $\bu(x,t)$. We consider two different types of available data. 
\medskip 

\noindent{\bf I. Multipopulation Data.} In the first case, we assume that we can measure  $\bu(x,0)$ for all possible initial conditions $\bm_0:=(m_{1,0}, \dots, m_{n,0})$, that is, we have data from the map:
\begin{equation}\label{EQ:Data 1}
    \cM_{(\bF,\bG)}: \bm_0 \mapsto \bu(x,t)|_{t=0}\,.
\end{equation}
Let $\cM_{(\bF^1, \bG^1)}$ and $\cM_{(\bF^2, \bG^2)}$ be the data corresponding to the cost function pairs $(\bF^1, \bG^1)$ and $(\bF^2, \bG^2)$, respectively. We will establish the following uniqueness result:
\begin{equation}\label{EQ:Unique Full}
    \cM_{(\bF^1,\bG^1)}=\cM_{(\bF^2,\bG^2)}\quad  \text{implies}\quad (\bF^1,\bG^1)=(\bF^2,\bG^2),
\end{equation}
under appropriate assumptions that we will specify later.

\medskip 

\noindent{\bf II. Single-population Data.} In the second scenario, we assume that we can measure the value function at $t=0$ for only one given population, say, the $i$-th population. That is, we have data given by the map:
\begin{equation}\label{EQ:Data 2}
    \wt\cM_{(\bF,\bG)}: \bm_0\mapsto u_i(x,t)|_{t=0}, \qquad \mbox{for a given}\ i\in\cI:=\{1, \dots, n\}\,.
\end{equation}
Therefore, we have only partial information on $\bu(x,0)$. With such data, it is not expected that we can reconstruct the cost functions in its fully general form. We will, therefore, have to put additional constraints on the form of the cost functions to be reconstructed. 
We will consider two special classes of $(\bF, \bG)$. In the first class, all populations share the same cost functions, that is,
\begin{equation}\label{EQ:Class 1}
    \cC_1:=\left\{(\bF, \bG)\mid \bF=\big(F(x,\bm), \dots, F(x,\bm)\big), \ \bG=\big(G(x,\bm), \dots, G(x,\bm)\big)\right\}\,.
\end{equation}
In the second class, the cost functions are state-independent, that is
\begin{equation}\label{EQ:Class 2}
    \cC_2:=\left\{(\bF, \bG)\mid \bF=\big(F_1(\bm), \dots, F_n(\bm)\big), \ \bG=\big(G_1(\bm), \dots, G_n(\bm)\big)\right\}\,.
\end{equation}
For either class of the cost function pairs, we will show that one can uniquely reconstruct $(\bF, \bG)$ from data encoded in $\wt \cM_{(\bF,\bG)}$, that is, for $\cC=\cC_1$ or $\cC=\cC_2$,
\begin{equation}\label{EQ:Unique Single}
    \forall (\bF^1,\bG^1),(\bF^2,\bG^2)  \in \cC, \quad \wt\cM_{(\bF^1,\bG^1)}=\wt\cM_{(\bF^2,\bG^2)},\quad \text{implies}\quad (\bF^1,\bG^1)=(\bF^2,\bG^2)\,.
\end{equation}
While this result seems less surprising for cost function pairs in class $\cC_1$, it is quite nontrivial for cost function pairs in class $\cC_2$ as such cost functions are population-dependent while our data are only measured on a single population.

Inverse problems to mean field game models have attracted significant attention recently. On the computational side, numerical inversion algorithms have been developed for the reconstruction of different parameters in MFG models in applications~\cite{AgLeFuNu-JCP22, ChFuLiNuOs-IP23,DiLiOsYi-JSC22,FuLiOsLi-JCP23,KlLiYa-AMO24,LiJaLiNuOs-SIAM21,LeLiTeLiOs-SIAM21,ReSoTo-arXiv24,YuLaLiOs-JCP23,YuXiChLa-IP24}. Extensive numerical experiments have been performed to demonstrate the effectiveness of the methods as well as to gain insights on the mathematical properties of the inverse problems. On the theoretical side,  uniqueness and stability results have been developed for both linearized and fully nonlinear inverse problems with different data types; see, for instance,~\cite{ImLiYa-AML23,ImLiYa-IPI24,ImYa-MMAS23,KlLiLi-JIIP24,LiZh-arXiv23,LiMoZh-IP23,ReSoWaZh-IP24} and references therein. 

To the best of our knowledge, all the existing theoretical results on inverse problems for mean field game systems are done for single-population systems. This is the main motivation for the current work. Building on the classical multilinearization technique for handling nonlinear models, which has been successfully applied to mean field game inverse problems~\cite{DiLiZh-SIAM25,LiZh-arXiv23,LiMoZh-IP23} as well as many other inverse problems for nonlinear partial differential equations (see, for instance,~\cite{AsZh-JDE21,Isakov-ARMA93,LaLiLiSa-JMPA21,KrUh-PAMS20,HaLi-NA23,Kian-Nonlinearity23,ReSo-IP24} and references therein for some random sample publications), we derive uniqueness theories of the form~\eqref{EQ:Unique Full} and~\eqref{EQ:Unique Single} for the multipopulation MFG model~\eqref{EQ:MFG MultiPop}. 


The rest of the paper is organized as follows. We first introduce the notations used in this paper in~\Cref{SEC:Prelim}. In~\Cref{SEC:Forward}, we recall the local well-posedness results of the mean field game model~\eqref{EQ:MFG MultiPop} and introduce the multilinearization of the model. The main uniqueness results of the paper are presented in~\Cref{SEC:Recon Full,SEC:Recon Partial}. While the results in~\Cref{SEC:Recon Full} can be seen as generalizations of the uniqueness results of single-population models, those in~\Cref{SEC:Recon Partial} show some unique features of the multipopulation case. Concluding remarks and discussions are then offered in~\Cref{SEC:Concluding}.
 
\section{Setup of notations}
\label{SEC:Prelim}

Let us first fix some of the standard notations we will use throughout the paper. We follow the conventions in~\cite{LiMoZh-IP23} as we will cite a well-posedness result from that work.
 
We consider MFGs on the $d$-dimensional torus $\bbT^d$. Notice that any function defined on $\bbT^d$ should be $(1,1,\dots,1)$-periodic with respect to $x$. In other words, it should be 1-periodic with respect to $x_i$, $i=1,2,\dots,d$.

We use the notation $\bbN$ to denote the set of all nonnegative integers. Given a $d$-dimensional multi-index $\balpha=(\alpha_1,\dots,\alpha_d)\in\bbN^d$ and $x=(x_1,\dots,x_d)\in\bbR^d$, we take the following standard notations of multi-index:
\[
    x^{\balpha}=x_1^{\alpha_1}x_2^{\alpha_2}\cdots x_d^{\alpha_d},\ D^{\balpha}=\partial_{x_1}^{\alpha_1}\partial_{x_2}^{\alpha_2}\cdots \partial_{x_d}^{\alpha_d},\ \balpha!=\alpha_1!\alpha_2!\cdots \alpha_d!,\ |\balpha|=\sum_{i=1}^d \alpha_i.
\]

For $k\in\bbN$ and $\alpha\in(0,1)$, the H\"older space $C^{k+\alpha}(\bbT^d)$ is defined as the following subspace of $C^{k}(\bbT^d)$: $u\in C^{k+\alpha}(\bbT^d)$ if and only if $D^{\balpha}u$ exists and is $\alpha$-H\"older continuous for all $|\balpha|\le k$. The norm of $C^{k+\alpha}(\bbT^d)$ is defined as
\begin{equation} 
    \|u\|_{C^{k+\alpha}(\bbT^d)}:=\sum_{|\balpha|\le k}\|D^{\balpha}u\|_{L^\infty(\bbT^d)}+\sum_{|\balpha|=k}\sup_{x\neq y}\frac{|D^{\balpha}u(x)-D^{\balpha}u(y)|}{|x-y|^\alpha}.
\end{equation}

We denote by $Q:=\bbT^d\times[0,T]$ the time-space domain. The H\"older space $C^{k+\alpha,\frac{k+\alpha}{2}}(Q)$ is defined similarly to $C^{k+\alpha}(\bbT^d)$. We say $u\in C^{k+\alpha,\frac{k+\alpha}{2}}(Q)$ if and only if $D^{\balpha}\partial_t^j u$ exists and is $\alpha$-H\"older continuous in $x$ and $\frac{\alpha}{2}$-H\"older continuous in $t$ for all $\balpha\in\bbN^d$, $j\in\bbN$ with $|\balpha|+2j\le k$. The norm of $C^{k+\alpha,\frac{k+\alpha}{2}}(Q)$ is defined as
\begin{equation} 
    \|u\|_{C^{k+\alpha,\frac{k+\alpha}{2}}(Q)}:=\sum_{|\balpha|+2j\le k}\|D^{\balpha}\partial_t^j u\|_{L^\infty(Q)}+\sum_{|\balpha|+2j=k}\sup_{(x,t)\neq(y,s)}\frac{|D^{\balpha}\partial_t^j u(x,t)-D^{\balpha}\partial_t^j u(y,s)|}{|x-y|^{\alpha}+|t-s|^{\frac{\alpha}{2}}}.
\end{equation}

Since we need to treat vector-valued functions in multipopulation MFG system~\eqref{EQ:MFG MultiPop}, we give the definition of corresponding norms here. The H\"older norm of a vector-valued function $\bu=(u_1,u_2,\dots,u_n)$ is defined as
\begin{equation}\label{EQ:multi-norm}
    \begin{aligned}
        \|\bu\|_{C^{k+\alpha}(\bbT^d)}&:=\sum_{i=1}^n \|u_i\|_{C^{k+\alpha}(\bbT^d)},\\
        \|\bu\|_{C^{k+\alpha,\frac{k+\alpha}{2}}(Q)}&:=\sum_{i=1}^n \|u_i\|_{C^{k+\alpha,\frac{k+\alpha}{2}}(Q)}.
    \end{aligned}
\end{equation}
 
Following~\cite{LiMoZh-IP23}, we assume that the cost functions $F_i$ and $G_i$ ($1\le i\le n$) admit a generic power series representation of the form, $\bbeta$ being an $n$-dimensional multi-index,
\begin{equation}\label{EQ:power expansion}        F_i(x,\bz)=\sum_{\bbeta\in\bbN^n,|\bbeta|\ge 1} F_i^{(\bbeta)}(x)\frac{\bz^\bbeta}{\bbeta!}\,.
\end{equation}
Clearly, functions satisfying the conditions in the following definition admit this power series representation.
\begin{definition}
    We say a function $U(x,\bz):\bbT^d\times\bbC^n\to\bbC$ is admissible or in class $\cA$, that is, $U\in\cA$, if it satisfies the following conditions:\\[1ex]
    (i) The map $\bz\mapsto U(\cdot,\bz)$ is holomorphic with value in $C^\alpha(\bbT^d)$ for some $\alpha\in(0,1)$;\\[1ex]
    (ii) $U(x,\bzero)=0$ for all $x\in\bbT^d$.
\end{definition}
We use the standard notation
\begin{equation}\label{EQ:A}
    \bF=(F_1, \cdots, F_n) \in\cA^n \iff F_i\in\cA, \forall i\ge 1\,.
\end{equation}
Throughout the rest of the paper, we assume the cost functions $\bF\in\cA^n$ and $\bG\in\cA^n$.

\section{The forward problem and its linearization}
\label{SEC:Forward}

We now recall the necessary results on the MFG system~\eqref{EQ:MFG MultiPop} for our analysis in the next sections. We first review a local well-posedness result that provides the theoretical foundation needed for the multilinearization of the MFG system that our uniqueness proofs are based on.

\subsection{Well-posedness of the forward problem}
\label{SUBSEC:Wellposedness}

We first state the following well-posedness result for the linearized MFG system as a preliminary.
\begin{lemma}\label{LEMMA:well-posedness}
    Suppose $f,\tilde{f}\in C^{\alpha,\frac{\alpha}{2}}(Q)$, $F_k\in C^{\alpha}(\bbT^d)$, and $g,\tilde{g},G_k\in C^{2+\alpha}(\bbT^d)$ for some $\alpha\in(0,1)$, then the system
    \begin{equation}\label{EQ:linearized system}
        \left\{\begin{array}{>{\displaystyle}lc}
            -\partial_t u_i(x,t)-\Delta u_i(x,t)=\sum_{k=1}^n F_k(x)m_k(x,t)+f(x,t), &(x,t)\in\bbT^d\times(0,T),\\
            \partial_t m_i(x,t)-\Delta m_i(x,t)=\tilde{f}(x,t), &(x,t)\in\bbT^d\times(0,T),\\
            u_i(x,T)=\sum_{k=1}^n G_k(x)m_k(x,T)+g(x),\ m(x,0)=\tilde{g}(x), &x\in\bbT^d.
        \end{array}\right.
    \end{equation}
    admits a unique solution $(\bu,\bm)\in [C^{2+\alpha,1+\frac{\alpha}{2}}(Q)]^{2n}$.
\end{lemma}
This result can be seen as a corollary of the standard well-posedness result for linear parabolic equations, since we can first obtain $\bm\in[C^{2+\alpha,1+\frac{\alpha}{2}}(Q)]^{n}$ from the second equation of \eqref{EQ:linearized system}, then put these $m_k$ into the first equation to obtain $\bu\in[C^{2+\alpha,1+\frac{\alpha}{2}}(Q)]^{n}$. Therefore, the detailed proof is omitted here.

Based on the previous Lemma, one can prove the following theorem, which is vital for the linearization process in the next subsection.
\begin{theorem}\label{THM:holomorphic}
    For any given $\bF\in \cA^n$ and $\bG\in\cA^n$, the following results hold:

    (a) There exist constants $\delta,C>0$, such that for any
    \begin{equation}\nonumber
        \bm_0\in B_\delta\left([C^{2+\alpha}(\bbT^d)]^n\right):=\{\bm_0\in[C^{2+\alpha}(\bbT^d)]^n:\|\bm_0\|_{C^{2+\alpha}(\bbT^d)}\le\delta\},
    \end{equation}
    the multipopulation MFG system \eqref{EQ:MFG MultiPop} admits a solution $(\bu,\bm)$ which satisfies $\bu(x,t),\bm(x,t)\in[C^{2+\alpha,1+\frac{\alpha}{2}}(Q)]^n$, and
    \begin{equation}\nonumber
        \|(\bu,\bm)\|_{C^{2+\alpha,1+\frac{\alpha}{2}}(Q)}\le C\|\bm_0\|_{C^{2+\alpha}(\bbT^d)}.
    \end{equation}
    Furthermore, the solution $(\bu,\bm)$ is unique within the class
    \begin{equation}\nonumber
        \{(\bu,\bm)\in [C^{2+\alpha,1+\frac{\alpha}{2}}(Q)]^{2n}:\|(\bu,\bm)\|_{C^{2+\alpha,1+\frac{\alpha}{2}}(Q)}\le C\delta\}.
    \end{equation}

    (b) Consider the solution map
    \begin{equation}\nonumber
        S:B_\delta\left([C^{2+\alpha}(\bbT^d)]^n\right)\to [C^{2+\alpha,1+\frac{\alpha}{2}}(Q)]^{2n}
    \end{equation}
    defined by $S(\bm_0)=(\bu,\bm)$ where $(\bu,\bm)$ is the unique solution to system \eqref{EQ:MFG MultiPop} in (a). Then $S$ is holomorphic.
\end{theorem}
This local well-posedness result can be proved in the same manner as the proof of~\cite[Theorem 3.1]{LiMoZh-IP23} using the implicit function theorem for Banach spaces, even though the result of ~\cite[Theorem 3.1]{LiMoZh-IP23} is for a single-population MFG system. We will not reproduce the proof here. The most important point about this result is that the solution to the system is holomorphic with respect to $\bm_0(x)$ near $0$. This ensures that the linearization procedure in the next subsection can be performed safely.

\subsection{Multi-linearizations}
\label{SUBSEC:Linearization}

One of the main challenges in the analysis of MFG system~\eqref{EQ:MFG MultiPop} is its nonlinearity. The standard method to handle nonlinearity in forward models of inverse problems is the method of linearization~\cite{Isakov-ARMA93,Isakov-Book06}. Roughly speaking, this method utilizes differential data to simplify inverse problems for nonlinear models to inverse problems for simplified linear models; see~\cite{KrUh-PAMS20,Kian-Nonlinearity23,LuZh-IP24} and references therein for a sense of recent development on the subject. The linearization procedure is standard now and has been utilized in solving inverse problems for mean field game systems~\cite{LiMoZh-IP23,LiZh-arXiv22}.

Here, we develop a multi-linearization for the MFG system~\eqref{EQ:MFG MultiPop} with respect to $\bm_0(x)$. This rigorous justification of the linearization procedure requires the infinite differentiability of the solution map $S$ defined in~\Cref{THM:holomorphic}.

First, we introduce the basic setting. Consider the multipopulation MFG system \eqref{EQ:MFG MultiPop}, and let
\begin{equation}
    m_{i,0}(x;\varepsilon)=\sum_{\ell=1}^N \varepsilon_\ell f_\ell^i(x), \ \ \ N\ge 1
\end{equation}
where $f_\ell^i\in C^{2+\alpha}(\bbT^d)$ satisfies $f_\ell^i(x)\ge 0$, and $\varepsilon=(\varepsilon_1,\varepsilon_2,\dots,\varepsilon_N)\in\bbR_+^N$ with $|\varepsilon|=\sum_{\ell=1}^N|\varepsilon_\ell|$ small enough. Then there exists $\delta>0$, such that $\bm_0\in B_\delta\left([C^{2+\alpha}(\bbT^d)]^n\right)$. Let $(\bu(x,t;\varepsilon),\bm(x,t;\varepsilon))$ be the unique solution of system \eqref{EQ:MFG MultiPop} with respect to $\varepsilon$, which is admitted by \Cref{THM:holomorphic}. Then for $\bF,\bG\in\cA^n$, we have $(\bu(x,t;0),\bm(x,t;0))=(\bzero,\bzero)$.

By \Cref{THM:holomorphic}, we have the following expansion of $\bu(x,t;\varepsilon)$ and $\bm(x,t;\varepsilon)$:
\begin{equation}\label{EQ:expansion w.r.t. epsilon}
    \begin{aligned}
        u_i(x,t;\varepsilon)=\sum_{\ell=1}^N \varepsilon_\ell u_i^{(\ell)}(x,t)+\frac{1}{2}\sum_{\ell_1=1}^N\sum_{\ell_2=1}^N \varepsilon_{\ell_1}\varepsilon_{\ell_2}u_i^{(\ell_1,\ell_2)}(x,t)+o(\varepsilon^2),\\
        m_i(x,t;\varepsilon)=\sum_{\ell=1}^N \varepsilon_\ell m_i^{(\ell)}(x,t)+\frac{1}{2}\sum_{\ell_1=1}^N\sum_{\ell_2=1}^N \varepsilon_{\ell_1}\varepsilon_{\ell_2}m_i^{(\ell_1,\ell_2)}(x,t)+o(\varepsilon^2),
    \end{aligned}
\end{equation}
where we have omitted the higher-order terms for simplicity. Here the first-order coefficients $\bu^{(\ell)}$ and $\bm^{(\ell)}$ are determined by
\begin{equation}
    \begin{aligned}
        u_i^{(\ell)}:=\partial_{\varepsilon_\ell}u_i|_{\varepsilon=0}=\lim_{\varepsilon_\ell\to 0}\frac{u_i(x,t;\varepsilon_\ell {\mathfrak e}_\ell)}{\varepsilon_\ell}\,,\\        m_i^{(\ell)}:=\partial_{\varepsilon_\ell}m_i|_{\varepsilon=0}=\lim_{\varepsilon_\ell\to 0}\frac{m_i(x,t;\varepsilon_\ell \mathfrak e_\ell)}{\varepsilon_\ell}\,,
    \end{aligned}
\end{equation}
where $\mathfrak e_\ell$ is the $\ell$-th element in the standard basis of $\bbR^N$.

We can then construct a new linearized system for $(\bu^{(\ell)},\bm^{(\ell)})$. Here, we take $\ell=1$ as an example. By calculation, we have
\begin{equation}
    \begin{aligned}
        & -\partial_t u_i^{(1)}(x,t)-\Delta u_i^{(1)}(x,t)\\
        = {} & \lim_{\varepsilon\to 0}\frac{1}{\varepsilon_1}\left(-\partial_t u_i(x,t;\varepsilon)-\Delta u_i(x,t;\varepsilon)\right)\\
        = {} & \lim_{\varepsilon\to 0}\frac{1}{\varepsilon_1}\left(-\frac{1}{2}|\nabla u_i(x,t;\varepsilon)|^2+F_i(x,\bm(x,t;\varepsilon))\right)\\
        = {} & \sum_{k=1}^n F_i^{(\be_k)}(x) m_k^{(1)},
    \end{aligned}
\end{equation}
where $\be_k$ is the $n$-dimensional multi-index $\bbeta=(\beta_1, \dots, \beta_n)$ such that $\beta_j=\delta_{jk}$, and
\begin{equation}
    \begin{aligned}
        & \partial_t m_i^{(1)}(x,t)-\Delta m_i^{(1)}(x,t)\\
        = {} & \lim_{\varepsilon\to 0}\frac{1}{\varepsilon_1}\left(\partial_t m_i(x,t;\varepsilon)-\Delta m_i(x,t;\varepsilon)\right)\\
        = {} & \lim_{\varepsilon\to 0}\frac{1}{\varepsilon_1}\left({\rm div}(m(x,t;\varepsilon)\nabla u(x,t;\varepsilon))\right)\\
        = {} & 0.
    \end{aligned}
\end{equation}
Therefore we have that $(\bu^{(1)},\bm^{(1)})$ satisfies the following system:
\begin{equation}\label{EQ:MFG Lin Order 1}
    \left\{\begin{array}{>{\displaystyle}lc}
        -\partial_t u_i^{(1)}-\Delta u_i^{(1)}=\sum_{k=1}^n F_i^{(\be_k)}(x) m_k^{(1)}, &(x,t)\in\bbT^d\times(0,T),\\
        \partial_t m_i^{(1)}-\Delta m_i^{(1)}=0, &(x,t)\in\bbT^d\times(0,T),\\
        u_i^{(1)}(x,T)=\sum_{k=1}^n G_i^{(\be_k)}(x) m_k^{(1)}(x,T),\ m_i^{(1)}(x,0)=f_1^i(x), &x\in\bbT^d.
    \end{array}\right.
\end{equation}
This first-order linearized system will be our starting point in the proof of the uniqueness of the inverse problem. It highlights one of the main features of the multipopulation model we study in this paper: the existence of the coupling, in the form of a summation of contributions from different populations, between different populations. We will need to decouple the contributions from different populations in the inverse problem; see ~\Cref{SEC:Recon Full} and~\Cref{SEC:Recon Partial}.

Higher-order linearized systems can be derived in a similar way. The first-order coefficients $\bu^{(\ell_1,\ell_2)}$ and $\bm^{(\ell_1,\ell_2)}$ in \eqref{EQ:expansion w.r.t. epsilon} are determined by
\begin{equation}
    \begin{aligned}
        u_i^{(\ell_1,\ell_2)} :=\partial_{\varepsilon_{\ell_1}}\partial_{\varepsilon_{\ell_2}}u_i|_{\varepsilon=0} =\lim_{\varepsilon\to 0}\frac{u_i(x,t;\varepsilon_{\ell_1}{\mathfrak e}_{\ell_1}+\varepsilon_{\ell_2}{\mathfrak e}_{\ell_2})-\varepsilon_{\ell_1}u_i^{(\ell_1)}(x,t)-\varepsilon_{\ell_2}u_i^{(\ell_2)}(x,t)}{\frac{1}{2}\varepsilon_{\ell_1}\varepsilon_{\ell_2}},\\
        m_i^{(\ell_1,\ell_2)} :=\partial_{\varepsilon_{\ell_1}}\partial_{\varepsilon_{\ell_2}}m_i|_{\varepsilon=0} =\lim_{\varepsilon\to 0}\frac{m_i(x,t;\varepsilon_{\ell_1}{\mathfrak e}_{\ell_1}+\varepsilon_{\ell_2}{\mathfrak e}_{\ell_2})-\varepsilon_{\ell_1}m_i^{(\ell_1)}(x,t)-\varepsilon_{\ell_2}m_i^{(\ell_2)}(x,t)}{\frac{1}{2}\varepsilon_{\ell_1}\varepsilon_{\ell_2}}.
    \end{aligned}
\end{equation}

We can also construct corresponding system for $(\bu^{(\ell_1,\ell_2)},\bm^{(\ell_1,\ell_2)})$. Again, we take $(\ell_1,\ell_2)=(1,2)$ for example. By direct calculation, we have
\begin{equation}
    \begin{aligned}
        & -\partial_t u_i^{(1,2)}(x,t)-\Delta u_i^{(1,2)}(x,t)\\
        = {} & \partial_{\varepsilon_1}\partial_{\varepsilon_2}\left(-\frac{1}{2}|\nabla u_i(x,t;\varepsilon)|^2+F_i(x,\bm(x,t;\varepsilon))\right)|_{\varepsilon=0}\\
        = {} & -\nabla u_i^{(1)}\cdot\nabla u_i^{(2)}+\sum_{k=1}^n F_i^{(\be_k)}(x) m_k^{(1,2)}+\sum_{k=1}^n \sum_{p=1}^n F_i^{(\be_k+\be_p)}(x)m_k^{(1)}m_p^{(2)},
    \end{aligned}
\end{equation}
and
\begin{equation}
    \begin{aligned}
        & \partial_t m_i^{(1,2)}(x,t)-\Delta m_i^{(1,2)}(x,t)\\
        = {} & \partial_{\varepsilon_1}\partial_{\varepsilon_2}{\rm div}(m(x,t;\varepsilon)\nabla u(x,t;\varepsilon))|_{\varepsilon=0}\\
        = {} & {\rm div}(m_i^{(1)}\nabla u_i^{(2)}+m_i^{(2)}\nabla u_i^{(1)}).
    \end{aligned}
\end{equation}
Therefore we have $(\bu^{(1,2)},\bm^{(1,2)})$ satisfies the following system:
\begin{equation}\nonumber
    \left\{\begin{array}{>{\displaystyle}lc}
        -\partial_t u_i^{(1,2)}-\Delta u_i^{(1,2)}+\nabla u_i^{(1)}\cdot\nabla u_i^{(2)}\\
        \qquad=\sum_{k=1}^n F_i^{(\be_k)}(x) m_k^{(1,2)}+\sum_{k=1}^n \sum_{p=1}^n F_i^{(\be_k+\be_p)}(x)m_k^{(1)}m_p^{(2)}, &(x,t)\in\bbT^d\times(0,T),\\
        \partial_t m_i^{(1,2)}-\Delta m_i^{(1,2)}={\rm div}(m_i^{(1)}\nabla u_i^{(2)}+m_i^{(2)}\nabla u_i^{(1)}), &(x,t)\in\bbT^d\times(0,T),\\
        u_i^{(1,2)}(x,T)\\
        \qquad=\sum_{k=1}^n G_i^{(\be_k)}(x) m_k^{(1,2)}(x,T)+\sum_{k=1}^n \sum_{p=1}^n G_i^{(\be_k+\be_p)}(x)m_k^{(1)}m_p^{(2)}(x,T),\\
        m_i^{(1,2)}(x,0)=0, &x\in\bbT^d.
    \end{array}\right.
\end{equation}

Inductively, we can derive a sequence of systems by considering
\begin{equation}
    \begin{aligned}
        u_i^{(\ell_1,\ell_2,\dots,\ell_M)}=\partial_{\varepsilon_{\ell_1}}\partial_{\varepsilon_{\ell_2}}\cdots\partial_{\varepsilon_{\ell_M}}u_i(x,t;\varepsilon)|_{\varepsilon=0},\\
        m_i^{(\ell_1,\ell_2,\dots,\ell_M)}=\partial_{\varepsilon_{\ell_1}}\partial_{\varepsilon_{\ell_2}}\cdots\partial_{\varepsilon_{\ell_M}}m_i(x,t;\varepsilon)|_{\varepsilon=0}.
    \end{aligned}
\end{equation}
Notice that all these different order linearizations will be helpful in proving our main results.

\section{Reconstructing \texorpdfstring{$(\bF,\bG)$}{} from multipopulation data}
\label{SEC:Recon Full}

We are now in a position to state and prove our main results about the unique determination of cost functions. We start with the setup where we  have measured data on all populations, encoded in the map $\cM_{(\bF, \bG)}$ given in~\eqref{EQ:Data 1}

First, we state an auxiliary lemma as follows, which is standard and can be proved using integration by parts.

\begin{lemma}\label{LEMMA:Orthogonal relation}
    Let $u$ be the solution of the following equation:
    \begin{equation}
        \left\{\begin{array}{>{\displaystyle}lc}
            -\partial_t u(x,t)-\Delta u(x,t)=f(x,t), &(x,t)\in\bbT^d\times(0,T),\\
            u(x,T)=u_T(x),\ u(x,0)=0, &x\in\bbT^d.
        \end{array}\right.
    \end{equation}
    Let $w(x,t)$ be a solution of $\partial_t w(x,t)-\Delta w(x,t)=0$ in $\bbT^d\times(0,T)$, then we have
    \begin{equation}
        \int_Q f(x,t)w(x,t)dxdt=\int_{\bbT^d}u_T(x)w(x,T)dx.
    \end{equation}
\end{lemma}

\begin{proof}
    By direct calculation, we have
    \begin{equation}\nonumber
        \begin{aligned}
            \int_Q f(x,t)w(x,t)dxdt &=\int_Q (-\partial_t u(x,t)-\Delta u(x,t))w(x,t)dxdt\\
            &=\left.\int_{\bbT^d}(u(x,t)w(x,t))dx\right|_0^T+\int_Q u(x,t)(\partial_t w(x,t)-\Delta w(x,t))dxdt\\
            &=\int_{\bbT^d}u_T(x)w(x,T)dx,
        \end{aligned}
    \end{equation}
    which has already proved the lemma.
\end{proof}

Here, we state our main results as follows, which shows that we can reconstruct the running cost $\bF$ and terminal cost $\bG$ from the measurement map $\cM$. The main idea is to consider different order Taylor coefficients of $\bF$ and $\bG$, and reconstruct them respectively using the multi-linearization method in \Cref{SUBSEC:Linearization}. The proof is similar to that in~\cite{LiMoZh-IP23} for a single-population MFG system.

\begin{theorem}\label{THM:full recon}
    Let $\bF^j, \bG^j\in\cA^n$ ($j=1,2$) and $\cM_{\bF^j,\bG^j}$ be the measurement map associated to the following system:
    \begin{equation}\label{EQ:MFG MultiPop j}
        \left\{\begin{array}{>{\displaystyle}lc}
            -\partial_t u_i^j(x,t)-\Delta u_i^j(x,t)+\frac{1}{2}|\nabla u_i^j(x,t)|^2=F_i^j(x,\bm^j(x,t)), &(x,t)\in\bbT^d\times(0,T),\\
            \partial_t m_i^j(x,t)-\Delta m_i^j(x,t)-{\rm div}(m_i^j(x,t)\nabla u_i^j)=0, &(x,t)\in\bbT^d\times(0,T)\\
            u_i^j(x,T)=G_i^j(x,\bm(x,T)),\ m_i^j(x,0)=m_{i,0}(x), &x\in \bbT^d.
        \end{array}\right.
    \end{equation}
    If there exists $\delta>0$, such that for any $\bm_0\in B_\delta\left([C^{2+\alpha}(\bbT^d)]^n\right)$, one has
    \[
        \cM_{\bF^1,\bG^1}(\bm_0)=\cM_{\bF^2,\bG^2}(\bm_0),
    \]
    then it holds that
    \[
        (\bF^1(x,\bz),\bG^1(x,\bz))=(\bF^2(x,\bz),\bG^2(x,\bz))\ \text{in}\ \bbT^d\times\bbR^n.
    \]
\end{theorem}

\begin{proof}
    We only need to prove that every term in the Taylor expansion of $F_i^j$ and $G_i^j$ coincide, that is:
    \begin{equation}\label{EQ:Unique expansion F}
        F_i^{1(\balpha)}(x)=F_i^{2(\balpha)}(x),\ G_i^{1(\balpha)}(x)=G_i^{2(\balpha)}(x),\ \balpha\in\bbZ_+^n,\ i=1,2,\dots,n.
    \end{equation}
    To this purpose, we divide our proof into three steps. For simplicity and unity of notations, we will use
    \[
        F_i^{(k)}:=F_i^{(\be_k)},\ F_i^{(k,p)}:=F_i^{(\be_k+\be_p)}
    \]
    and so on to denote different orders of Taylor coefficients of $F$ in this proof and the rest of this paper.

    \textbf{Step I.} First, we implement the first order linearization to system \eqref{EQ:MFG MultiPop j} and show that \eqref{EQ:Unique expansion F} holds for all $|\balpha|=1$. In this way, we can obtain the following system:
    \begin{equation}
        \left\{\begin{array}{>{\displaystyle}lc}
            -\partial_t u_i^{j(1)}-\Delta u_i^{j(1)}=\sum_{k=1}^n F_i^{j(k)}(x) m_k^{j(1)}, &(x,t)\in\bbT^d\times(0,T),\\
            \partial_t m_i^{j(1)}-\Delta m_i^{j(1)}=0, &(x,t)\in\bbT^d\times(0,T),\\
            u_i^{j(1)}(x,T)=\sum_{k=1}^n G_i^{j(k)}(x) m_k^{j(1)}(x,T),\ m_i^{j(1)}(x,0)=f_1^i(x), &x\in\bbT^d.
        \end{array}\right.
    \end{equation}
    Then obviously we have $\bm^{1(1)}(x,t)=\bm^{2(1)}(x,t)=:\bm^{(1)}(x,t)$.
    
    Let $\overline{\bu}^{(1)}=\bu^{1(1)}-\bu^{2(1)}$, then $\overline{\bu}^{(1)}$ satisfies
    \begin{equation}
        \left\{\begin{array}{>{\displaystyle}lc}
            -\partial_t \overline{u}_i^{(1)}-\Delta \overline{u}_i^{(1)}=\sum_{k=1}^n (F_i^{1(k)}(x)-F_i^{2(k)}(x))m_k^{(1)}, &(x,t)\in\bbT^d\times(0,T),\\
            \overline{u}_i^{(1)}(x,T)=\sum_{k=1}^n (G_i^{1(k)}(x)-G_i^{2(k)}(x))m_k^{(1)}(x,T), &x\in\bbT^d,\\
            \overline{u}_i^{(1)}(x,0) = 0, &x\in\bbT^d,
        \end{array}\right.
    \end{equation}
    where the last equation comes from the data. By \Cref{LEMMA:Orthogonal relation}, we have
    \begin{equation}\label{EQ:First order orthogonal}
        \int_Q \sum_{k=1}^n (F_i^{1(k)}-F_i^{2(k)})m_k^{(1)}w(x,t) dxdt=\int_{\bbT^d}\sum_{k=1}^n (G_i^{1(k)}-G_i^{2(k)})m_k^{(1)}w(x,T)dx
    \end{equation}
    for all $w(x,t),m_k^{(1)}(x,t)$ that are solutions of the heat equation $\partial_t v(x,t)-\Delta v(x,t)=0$ in $Q$.
    
    Let $\xi_1,\xi_2\in\bbZ^d\backslash\{0\}$, $\xi=\xi_1+\xi_2$, and $M\in\bbZ_+$. Let
    \begin{equation}
        w(x,t)=e^{-4\pi^2|\xi_1|^2 t+2\pi \mathfrak i\xi_1\cdot x},
    \end{equation}
    \begin{equation}
        m_1^{(1)}(x,t)=e^{-4\pi^2|\xi_2|^2 t+2\pi \mathfrak i\xi_2\cdot x}+M.
    \end{equation}
    and $m_k^{(1)}(x,t)=0$ for all $k\neq 1$. Here, the value of $M$ ensures the nonnegativity of $f_1^1(x) = m_1^{(1)}(x,0)$, and therefore the nonnegativity of $m_1$. Denote
    \begin{equation}
        \begin{aligned}
            a_\xi:=\int_{\bbT^d}(F_i^{1(1)}-F_i^{2(1)})e^{2\pi \mathfrak i\xi\cdot x}dx,\\
            b_\xi:=\int_{\bbT^d}(G_i^{1(1)}-G_i^{2(1)})e^{2\pi \mathfrak i\xi\cdot x}dx,
        \end{aligned}
    \end{equation}
    then through direct calculations, \eqref{EQ:First order orthogonal} leads to
    \begin{equation}
        \frac{1-e^{-4\pi^2(|\xi_1|^2+|\xi_2|^2)T}}{4\pi^2(|\xi_1|^2+|\xi_2|^2)}a_\xi+M\frac{1-e^{-4\pi^2|\xi_1|^2 T}}{4\pi^2|\xi_1|^2}a_{\xi_1}=e^{-4\pi^2(|\xi_1|^2+|\xi_2|^2)T}b_\xi+M e^{-4\pi^2|\xi_1|^2 T}b_{\xi_1}.
    \end{equation}
    By taking $M=1,2$ and subtracting the corresponding equations, we have
    \begin{equation}
        \frac{1-e^{-4\pi^2(|\xi_1|^2+|\xi_2|^2)T}}{4\pi^2(|\xi_1|^2+|\xi_2|^2)}a_\xi=e^{-4\pi^2(|\xi_1|^2+|\xi_2|^2)T}b_\xi.
    \end{equation}

    Notice that for any $\xi\in\bbZ^d$, there exist $\xi_1,\xi_2,\xi_1',\xi_2'\in\bbZ^d\backslash\{0\}$, such that $\xi=\xi_1+\xi_2=\xi_1'+\xi_2'$, but $|\xi_1|^2+|\xi_2|^2\neq |\xi_1'|^2+|\xi_2'|^2$. Therefore, we have $a_\xi=b_\xi=0$ for all $\xi\in\bbZ^d$. This implies the Fourier series of $F_i^{1(1)}-F_i^{2(1)}$ and $G_i^{1(1)}-G_i^{2(1)}$ are both 0, hence $F_i^{1(1)}=F_i^{2(1)}$ and $G_i^{1(1)}=G_i^{2(1)}$.

    The same argument holds for all $k$, therefore
    \begin{equation}\label{EQ:All i}
        F_i^{1(k)}(x)=F_i^{2(k)}(x),\ G_i^{1(k)}(x)=G_i^{2(k)}(x),\ i,k=1,2,\dots,n.
    \end{equation}
    Then, by the uniqueness result of the linearized system shown in \Cref{LEMMA:well-posedness}, we have $\bu^{1(1)}(x,t)=\bu^{2(1)}(x,t)$.

    \textbf{Step II.} We continue to implement the second order linearization to system \eqref{EQ:MFG MultiPop j} and obtain:
    \begin{equation}
        \left\{\begin{array}{>{\displaystyle}lc}
            -\partial_t u_i^{j(1,2)}-\Delta u_i^{j(1,2)}+\nabla u_i^{j(1)}\cdot\nabla u_i^{j(2)}\\
            \qquad=\sum_{k=1}^n F_i^{j(k)}(x) m_k^{j(1,2)}+\sum_{k=1}^n \sum_{p=1}^n F_i^{j(k,p)}(x)m_k^{j(1)}m_p^{j(2)}, &(x,t)\in\bbT^d\times(0,T),\\
            \partial_t m_i^{j(1,2)}-\Delta m_i^{j(1,2)}={\rm div}(m_i^{j(1)}\nabla u_i^{j(2)}+m_i^{j(2)}\nabla u_i^{j(1)}), &(x,t)\in\bbT^d\times(0,T),\\
            u_i^{j(1,2)}(x,T)\\
            \qquad=\sum_{k=1}^n G_i^{j(k)}(x) m_k^{j(1,2)}(x,T)+\sum_{k=1}^n \sum_{p=1}^n G_i^{j(k,p)}(x)m_k^{j(1)}m_p^{j(2)}(x,T),\\
            m_i^{j(1,2)}(x,0)=0, &x\in\bbT^d.
        \end{array}\right.
    \end{equation}
    By Step I, we have
    \[
        \bu^{1(1)}(x,t)=\bu^{2(1)}(x,t),\ \bu^{1(2)}(x,t)=\bu^{2(2)}(x,t),
    \]
    and
    \[
        \bm^{1(1)}(x,t)=\bm^{2(1)}(x,t),\ \bm^{1(2)}(x,t)=\bm^{2(2)}(x,t).
    \]
    Therefore we have
    \[
        \bm^{1(1,2)}(x,t)=\bm^{2(1,2)}(x,t).
    \]

    Let $\overline{\bu}^{(2)}=\bu^{1(1,2)}-\bu^{2(1,2)}$, then $\overline{\bu}^{(1,2)}$ satisfies
    \begin{equation}
        \left\{\begin{array}{>{\displaystyle}lc}
            -\partial_t \overline{u}_i^{(2)}-\Delta \overline{u}_i^{(2)}=\sum_{k=1}^n \sum_{p=1}^n (F_i^{1(k,p)}(x)-F_i^{2(k,p)}(x))m_k^{(1)}m_p^{(2)}, &(x,t)\in\bbT^d\times(0,T),\\
            \overline{u}_i^{(2)}(x,T)=\sum_{k=1}^n \sum_{p=1}^n (G_i^{1(k,p)}(x)-G_i^{2(k,p)}(x))m_k^{(1)}m_p^{(2)}(x,T), &x\in\bbT^d,\\
            \overline{u}_i^{(2)}(x,0)=0, &x\in\bbT^d,
        \end{array}\right.
    \end{equation}
    where the last equation comes from the data. By Lemma \ref{LEMMA:Orthogonal relation}, we have
    \begin{equation}
        \int_Q \sum_{k=1}^n \sum_{p=1}^n (F_i^{1(k,p)}-F_i^{2(k,p)})m_k^{(1)}m_p^{(2)}w dxdt=\int_{\bbT^d}\sum_{k=1}^n \sum_{p=1}^n (G_i^{1(k,p)}-G_i^{2(k,p)})m_k^{(1)}m_p^{(2)}w(x,T)dx
    \end{equation}
    for all $w(x,t),m_k^{(1)}(x,t),m_p^{(2)}(x,t)$ that are solutions of the heat equation $\partial_t v(x,t)-\Delta v(x,t)=0$ in $\bbT^d$. Then, by a similar argument to Step I, we can show that
    \begin{equation}
        F_i^{1(k,p)}(x)=F_i^{2(k,p)}(x),\ G_i^{1(k,p)}(x)=G_i^{2(k,p)}(x),\ i,k,p=1,2,\dots,n.
    \end{equation}

    \textbf{Step III.} The only thing left is to complete the proof by induction. The process of induction is already implied in Step II, and here, we show it more clearly. Similar to the first and second order linearization, we can obtain the $s$-th ($s\ge 2$) order linearization of system \eqref{EQ:MFG MultiPop j}:
    \begin{equation}
        \left\{\begin{array}{>{\displaystyle}lc}
            -\partial_t u_i^{j(1,\dots,s)} - \Delta u_i^{j(1,\dots,s)}\\
            \qquad=\sum_{k=1}^n F_i^{j(k)}(x)m_k^{j(1,\dots,s)}\\
            \qquad\ +\sum_{k_1,\dots,k_s=1}^n F_i^{j(k_1,\dots,k_s)}(x) m_{k_1}^{j(1)}\cdots m_{k_s}^{j(s)}+l.o.t., &(x,t)\in \bbT^d \times (0,T),\\
            \partial_t m_i^{j(1,\dots,n)} - \Delta m_i^{j(1,\dots,n)} = l.o.t., &(x,t)\in \bbT^d \times (0,T),\\
            u_j^{(1,\dots,n)}(x,T)\\
            \qquad=\sum_{k=1}^n G_i^{j(k)}(x)m_k^{j(1,\dots,s)}(x,T)\\
            \qquad\ +\sum_{k_1,\dots,k_s=1}^n G_i^{j(k_1,\dots,k_s)}(x) m_{k_1}^{j(1)}\cdots m_{k_s}^{j(s)}(x,T)+{\rm l.o.t.},\\
            m_j^{(1,\dots,n)}(x,0) = 0, &x\in\bbT^d,
        \end{array}\right.
    \end{equation}
    where $l.o.t.$ stands for lower-order terms in the linearization process.
    
    Let $s\ge 2$, and assume that all lower-order terms are independent of $j$. Then we have $\bm^{1(1,\dots,n)}(x,t)=\bm^{2(1,\dots,n)}(x,t)$.
    
    Let $\overline{\bu}^{(s)}=\bu^{1(1,\dots,s)}-\bu^{2(1,\dots,s)}$, then $\overline{\bu}^{(s)}$ satisfies
    \begin{equation}
        \left\{\begin{array}{>{\displaystyle}lc}
            -\partial_t \overline{u}_i^{(s)}-\Delta \overline{u}_i^{(s)}=\sum_{k_1,\dots,k_s=1}^n (F_i^{1(k_1,\dots,k_s)}(x)-F_i^{2(k_1,\dots,k_s)}(x)) m_{k_1}^{(1)}\cdots m_{k_s}^{(s)}, &(x,t)\in\bbT^d\times(0,T),\\
            \overline{u}_i^{(s)}(x,T)=\sum_{k_1,\dots,k_s=1}^n (G_i^{1(k_1,\dots,k_s)}(x)-G_i^{2(k_1,\dots,k_s)}(x)) m_{k_1}^{(1)}\cdots m_{k_s}^{(s)}(x,T), &x\in\bbT^d.
        \end{array}\right.
    \end{equation}
    By Lemma \ref{LEMMA:Orthogonal relation}, we have
    \begin{equation}
        \begin{aligned}
            &\int_Q \sum_{k_1,\dots,k_s=1}^n (F_i^{1(k_1,\dots,k_s)}(x)-F_i^{2(k_1,\dots,k_s)}(x)) m_{k_1}^{(1)}\cdots m_{k_s}^{(s)}w dxdt\\
            &=\int_{\bbT^d}\sum_{k_1,\dots,k_s=1}^n (G_i^{1(k_1,\dots,k_s)}(x)-G_i^{2(k_1,\dots,k_s)}(x)) m_{k_1}^{(1)}\cdots m_{k_s}^{(s)}w(x,T)dx.
        \end{aligned}
    \end{equation}
    Then, following the same arguments in Steps I and II, we have
    \begin{equation}\label{EQ:s-order recon}
        F_i^{1(\balpha)}(x)=F_i^{2(\balpha)}(x),\ G_i^{1(\balpha)}(x)=G_i^{2(\balpha)}(x),\ |\balpha|=s,\ i=1,2,\dots,n.
    \end{equation}

    By induction, \eqref{EQ:s-order recon} holds for all $s\in\bbZ_+$. Therefore by the admissible conditions of $\bF(x,\bz)$ and $\bG(x,\bz)$, we have
    \[
        (\bF^1(x,\bz),\bG^1(x,\bz))=(\bF^2(x,\bz),\bG^2(x,\bz)),
    \]
    and the proof is completed.
\end{proof}

In the case when all the populations are decoupled, that is, $F_i$ and $G_i$ depend only on $m_i$, the inverse problem degenerates to that for the single-population case in~\cite{LiMoZh-IP23}, as we will be able to solve the problem for each population independently (since the coupling through the summation of $k$ in the first-order linearization~\eqref{EQ:MFG Lin Order 1} does not exist anymore).  

\section{Reconstructions from single-population data}
\label{SEC:Recon Partial}

We now consider the case where we have data only from one population of the system, that is, data encoded in the map $\wt \cM_{(\bF, \bG)}$ defined in ~\eqref{EQ:Data 2}. With such data, it is generally impossible to reconstruct the cost functions in very general forms. One obvious case is when the populations are all independent. In this case, measuring data from one population would not give us any information about a different population.

It turns out that we can still have unique reconstructions in some special scenarios.

\subsection{Population-independent cost functions}

The first case special case is when the cost functions $F_i$ and $G_i$ are the same for all populations, that is, $F_i(x, \bz)=F(x,\bz)$ and $G_i(x, \bz)=G(x, \bz)$ for all $1\le i\le n$. Therefore, $\bF\in\cC_1$ and $\bG\in\cC_1$ with $\cC_1$ the set defined in~\eqref{EQ:Class 1}.

The following result says that we can reconstruct the two cost functions from data given by the map $\wt \cM$ defined in \eqref{EQ:Data 2}.
\begin{theorem}\label{THM:part recon}
    Let $\bF^j\in\cA^n\cap\cC_1$, $\bG^j\in\cA^n\cap\cC_1$ ($j=1,2$), and $\wt \cM_{(\bF^j,\bG^j)}$ be the data associated to the following system:
    \begin{equation}
        \left\{\begin{array}{>{\displaystyle}lc}
            -\partial_t u_i^j(x,t)-\Delta u_i^j(x,t)+\frac{1}{2}|\nabla u_i^j(x,t)|^2=F^j(x,\bm^j(x,t)), &(x,t)\in\bbT^d\times(0,T),\\
            \partial_t m_i^j(x,t)-\Delta m_i^j(x,t)-{\rm div}(m_i^j(x,t)\nabla u_i^j)=0, &(x,t)\in\bbT^d\times(0,T)\\
            u_i^j(x,T)=G^j(x,\bm(x,T)),\ m_i^j(x,0)=m_{i,0}(x), &x\in \bbT^d.
        \end{array}\right.
    \end{equation}
    If there exists $\delta>0$, such that for any $\bm_0\in B_\delta\left([C^{2+\alpha}(\bbT^d)]^n\right)$, one has
    \[
        \wt \cM_{(\bF^1,\bG^1)}=\wt\cM_{(\bF^2,\bG^2)},
    \]
    then it holds that
    \[
        (\bF^1(x,\bz), \bG^1(x,\bz))=(\bF^2(x,\bz), \bG^2(x,\bz))\ \text{in}\ \bbT^d\times\bbR^n.
    \]
\end{theorem}

\begin{proof}
    The proof of this theorem follows the same framework as \Cref{THM:full recon}, so we will only outline the main steps and leave out most of the details in the proof.
    
    In this special case, the first-order linearization leads to the following system:
    \begin{equation}
        \left\{\begin{array}{>{\displaystyle}lc}
            -\partial_t u_i^{j(1)}-\Delta u_i^{j(1)}=\sum_{k=1}^n F^{j(k)}(x) m_k^{j(1)}, &(x,t)\in\bbT^d\times(0,T),\\
            \partial_t m_i^{j(1)}-\Delta m_i^{j(1)}=0, &(x,t)\in\bbT^d\times(0,T),\\
            u_i^{j(1)}(x,T)=\sum_{k=1}^n G^{j(k)}(x) m_k^{j(1)}(x,T),\ m_i^{j(1)}(x,0)=f_1^i(x), &x\in\bbT^d.
        \end{array}\right.
    \end{equation}
    The key observation here is that due to the fact that the cost functions are the same for all the populations, the right-hand side of the first equation, as well as the final condition in the third equation, of the system is the same for all populations. Therefore, the value function is also independent of the population. Therefore, we need only data from a single population in $\wt \cM$ defined in~\eqref{EQ:Data 2}. Without loss of generality, we assume the measurement is on the first population. We have the following orthogonality relation from the equation of $u_1^{j(1)}$ (which is a simplification of the orthogonality relation~\eqref{EQ:First order orthogonal}):
    \begin{equation}
        \int_Q (F^{1(k)}-F^{2(k)}) \sum_{k=1}^n m_k^{(1)}w(x,t) dxdt=\int_{\bbT^d} (G^{1(k)}-G^{2(k)}) \sum_{k=1}^n m_k^{(1)}w(x,T)dx,
    \end{equation}
    where $w(x,t)$ and $m_k^{(1)}(x,t)$ are solutions of the heat equation $\partial_t v(x,t)-\Delta v(x,t)=0$ in $Q$. The same argument in Step I of the proof for Theorem~\ref{THM:full recon} then gives us
    \begin{equation}
        F^{1(k)}(x)=F^{2(k)}(x),\ G^{1(k)}(x)=G^{2(k)}(x),\ k=1,2,\dots,n\,.
    \end{equation}
    This in turn implies that $\bu^{1(1)}(x,t)=\bu^{2(1)}(x,t)$.

    Then, by the same induction process established in Step II and III in the proof of \Cref{THM:full recon}, we can uniquely determine every Taylor coefficient of $F$ and $G$, and the proof is completed.
\end{proof}

\begin{remark}
    We emphasize again that ~\Cref{THM:part recon} is not the generalized case of \Cref{THM:full recon}, since we require all populations to share the same running cost $F$ and terminal cost $G$. In the general cases where different populations have different cost functions $F_i$ and $G_i$, data measured only on a given population, say, population $1$, are not enough to reconstruct the cost functions of that population. This is evidently seen from the fact that, in this case, the orthogonality relation~\eqref{EQ:First order orthogonal} can only be written for population $1$. One therefore can not reconstruct first-order Taylor coefficients for all populations as in~\eqref{EQ:All i}. Therefore, one can not conclude that $\bu^{1(1)}(x,t)=\bu^{2(1)}(x,t)$. The rest of the proof of~\Cref{THM:full recon} can not proceed as it was done.
\end{remark}

\subsection{State-independent cost functions}

The second special case of reconstruction from single-population data is the case where the cost functions $F_i$ are independent of the state variable $x$ as in~\eqref{EQ:Class 2}, so that the coefficients in the power series representation~\eqref{EQ:power expansion} are now constants, that is,
\[
    F_i(\bm)=\sum_{\balpha\in\bbN^n,|\balpha|\ge 1} F_i^{(\balpha)}\frac{\bm^\balpha}{\balpha!}\,,
\]
with $F_i^{(\balpha)}\in\bbR$. As in the proof of Theorem \ref{THM:full recon}, we will sometimes abuse notation and write $F_i^{(k)}$ to mean $F_i^{(\be_k)}$. 
    Such models of the cost functions can be viewed as simplified versions of the separable cost functions that can be written as a summation of a function of $x$ and a function of the population density; see, for instance, ~\cite{BaFe-NoDEA16,BaPr-SIAM14,GoPiSa-ESAIM16,LaWo-TRB-2011,PiVo-IUMJ17} for examples of such cost function models.

For simplicity, we assume that $G_i=0$. Then the system \eqref{EQ:MFG MultiPop} becomes:
\begin{equation}\label{EQ:MFG MultiPop simpler}
    \left\{\begin{array}{>{\displaystyle}lc}
        -\partial_t u_i(x,t)-\Delta u_i(x,t)+\frac{1}{2}|\nabla u_i(x,t)|^2=F_i(\bm(x,t)), &(x,t)\in\bbT^d\times(0,T),\\
        \partial_t m_i(x,t)-\Delta m_i(x,t)-{\rm div}(m_i(x,t)\nabla u_i)=0, &(x,t)\in\bbT^d\times(0,T),\\
        u_i(x,T)=0,\ m_i(x,0)=m_{i,0}(x), &x\in \bbT^d,
    \end{array}\right.
\end{equation}
and in the following, we change the notation of partial measurements $\wt\cM_{(\bF,\bG)}$ to $\wt\cM_{\bF}$.

In general cases, we are not able to reconstruct information in other populations from measurements only on a subset of the populations, even in this simplified setting. Consider $F_i(\bm)=F_i(m_i)$. In this case, it is only possible to reconstruct one component of $\bF(\bm)$ from the partial measurement $\wt\cM_{\bF}$. However, our following theorem shows we are able to reconstruct the whole $\bF(\bm)$ from $\wt\cM_{\bF}$ if some mild additional assumptions on the coupling mechanism are satisfied.

\begin{theorem}\label{THM:state-indep recon}
    Let $\wt\cM_{\bF^j}$ ($j=1,2$) be the single-population measurement map associated with system \eqref{EQ:MFG MultiPop simpler} with $\bF^j\in\cA^n\cap \cC_2$. If there exists $\delta>0$, such that for any $\bm_0\in B_\delta\left([C^{2+\alpha}(\bbT^d)]^n\right)$, one has
    \[
        \wt\cM_{\bF^1}(\bm_0)=\wt\cM_{\bF^2}(\bm_0),
    \]
    then it holds that
    \[
        F_1^1(\bz)=F_1^2(\bz)\ \text{in}\ \bbR^n.
    \]
    Furthermore, if $F_1^{(k)}\neq 0$, $k=1,2,\dots,n$, then we have
    \[
        \bF^1(\bz)=\bF^2(\bz)\ \text{in}\ \bbR^n.
    \]
\end{theorem}
\begin{proof}
    Without loss of generality, we assume that the data we measured is on the population $1$ (that is, $i=1$ in the definition~\eqref{EQ:Data 2}). We begin with the reconstruction of $F_1$. Notice that if we take $m_{i,0}(x)=c_i$, then the solution of system \eqref{EQ:MFG MultiPop simpler} is
    \[
        (u_i(x,t),m_i(x,t))=(u_i(t),c_i),
    \]
    where
    \[
        u_i(t)=(T-t)F_i(\bc).
    \]
    Therefore the data $u_1(x,0)$ uniquely determines $F_1(\bc)$ for all $\|\bc\|=\sum_{i=1}^n|c_i|\le\delta$. Then by the assumption that $F_1$ is holomorphic, we have $F_1^1(\bz)=F_1^2(\bz)$, $\bz\in\bbR^n$, and equivalently $F_1^{1(\balpha)}=F_1^{2(\balpha)}$ for all $\balpha\in\bbN^d$.

    Now we assume $F_1^{(k)}\neq 0$, $k=1,2,\dots,n$. We take the following notations: $\phi_\xi(x)=e^{2\pi i\xi\cdot x}$, $\psi_\lambda(t)=e^{4\pi^2\lambda t}$, where $\xi\in\bbZ^d$, $\lambda\in\bbR$. Now, we implement detailed calculations for the first and second-order linearizations. The first-order linearization leads to the following system (for simplicity, we will sometimes omit the index $j=1,2$):
    \begin{equation*}
        \left\{\begin{array}{>{\displaystyle}lc}
            -\partial_t u_i^{(1)}-\Delta u_i^{(1)}=\sum_{k=1}^n F_i^{(k)} m_k^{(1)}, &(x,t)\in\bbT^d\times(0,T),\\
            \partial_t m_i^{(1)}-\Delta m_i^{(1)}=0, &(x,t)\in\bbT^d\times(0,T),\\
            u_i^{(1)}(x,T)=0,\ m_i^{(1)}(x,0)=f_1^i(x), &x\in\bbT^d.
        \end{array}\right.
    \end{equation*}
    Take $f_1^i(x)=c_i\phi_{\xi}(x)$, then the solution of the second equation is
    \begin{equation*}
        m_i^{(1)}(x,t)=c_i\psi_{-|\xi|^2}(t)\phi_{\xi}(x)=c_i e^{-4\pi^2|\xi|^2 t+2\pi i\xi\cdot x}.
    \end{equation*}
    The solution of the first equation is therefore
    \begin{equation*}
        u_i^{(1)}(x,t)=-\sum_{k=1}^n \frac{F_i^{(k)}c_k}{4\pi^2|\xi|^2}H^1_\xi(t)\phi_{\xi}(x),
    \end{equation*}
    where
    \begin{equation*}
        H^1_\xi(t):=\frac{1}{2}\left(-\psi_{-|\xi|^2}(t)+e^{-4\pi^2|\xi|^2 \cdot 2T}\psi_{|\xi|^2}(t)\right).
    \end{equation*}

    Then we consider the second-order linearization:
    \begin{equation*}
        \left\{\begin{array}{>{\displaystyle}lc}
            -\partial_t u_i^{(1,2)}-\Delta u_i^{(1,2)}+\nabla u_i^{(1)}\cdot\nabla u_i^{(2)}\\
            \qquad=\sum_{k_1=1}^n F_i^{(k_1)} m_{k_1}^{(1,2)}+\sum_{k_1=1}^n \sum_{k_2=1}^n F_i^{(k_1,k_2)}m_{k_1}^{(1)}m_{k_2}^{(2)}, &(x,t)\in\bbT^d\times(0,T),\\
            \partial_t m_i^{(1,2)}-\Delta m_i^{(1,2)}={\rm div}(m_i^{(1)}\nabla u_i^{(2)}+m_i^{(2)}\nabla u_i^{(1)}), &(x,t)\in\bbT^d\times(0,T),\\
            u_i^{(1,2)}(x,T)=0,\ m_i^{(1,2)}(x,0)=0, &x\in\bbT^d.
        \end{array}\right.
    \end{equation*}
    Take $f_2^{r_2}(x)=1$ for a fixed $r_2\in\{1,2,\dots,n\}$, and $f_2^k(x)=0$ for $k\neq r_2$, then we have $m_i^{(2)}(x,t)=\delta_{ir_2}$, and $\nabla u_i^{(2)}=0$. Therefore, the above system becomes
    \begin{equation}\label{EQ:Second order, x-independent}
        \left\{\begin{array}{>{\displaystyle}lc}
            -\partial_t u_i^{(1,2)}-\Delta u_i^{(1,2)}=\sum_{k_1=1}^n F_i^{(k_1)} m_{k_1}^{(1,2)}+\sum_{k_1=1}^n F_i^{(k_1,r_2)}m_{k_1}^{(1)}, &(x,t)\in\bbT^d\times(0,T),\\
            \partial_t m_i^{(1,2)}-\Delta m_i^{(1,2)}=\delta_{ir_2}\Delta u_i^{(1)}, &(x,t)\in\bbT^d\times(0,T),\\
            u_i^{(1,2)}(x,T)=0,\ m_i^{(1,2)}(x,0)=0, &x\in\bbT^d.
        \end{array}\right.
    \end{equation}

    By direct calculation, we have
    \begin{equation*}
        \Delta u_i^{(1)}(x,t)=\sum_{k=1}^n (F_i^{(k)}c_k) H^1_\xi(t)\phi_{\xi}(x).
    \end{equation*}
    Hence for the second equation, we have $m_i^{(1,2)}(x,t)\equiv 0$, $i\neq r_2$, and
    \begin{equation*}
        m_{r_2}^{(1,2)}(x,t)=\sum_{k=1}^n F_{r_2}^{(k)}c_k H^2_\xi(t)\phi_{\xi}(x),
    \end{equation*}
    where
    \begin{equation*}
        H^2_\xi(t):=\frac{1}{2}\left[-(t-T)\psi_{-|\xi|^2}(t)+\frac{e^{-4\pi^2|\xi|^2 \cdot 2T}}{8\pi^2|\xi|^2}\left(\psi_{|\xi|^2}(t)-e^{4\pi^2|\xi|^2 \cdot 2T}\psi_{-|\xi|^2}(t)\right)\right].
    \end{equation*}

    Now we only take one $c_{r_1}=1$, and let $c_k=0$ for $k\neq r_1$. Then we finally get
    \begin{equation*}
        -\partial_t u_i^{(1,2)}-\Delta u_i^{(1,2)}=F_i^{(r_2)}F_{r_2}^{(r_1)}H^2_\xi(t)\phi_{\xi}(x)+F_i^{(r_1,r_2)}\psi_{-|\xi|^2}(t)\phi_{\xi}(x)
    \end{equation*}
    for all $\xi\in\bbZ^d$. Take $i=1$, and let $\overline{u}_1^{(1,2)}=u_1^{1(1,2)}-u_1^{2(1,2)}$. Since all $F_1^{(\balpha)}$ are already uniquely determined, $\overline{u}_1^{(1,2)}$ satisfies the following system:
    \begin{equation*}
        \left\{\begin{array}{>{\displaystyle}lc}
            -\partial_t \overline{u}_1^{(1,2)}-\Delta \overline{u}_1^{(1,2)}=F_1^{(r_2)}(F_{r_2}^{1(r_1)}-F_{r_2}^{2(r_1)})H^2_\xi(t)\phi_{\xi}(x), &(x,t)\in\bbT^d\times(0,T),\\
            \overline{u}_1^{(1,2)}(x,T)=\overline{u}_1^{(1,2)}(x,0) = 0, &x\in\bbT^d,
        \end{array}\right.
    \end{equation*}
    Take $w(x,t):=\psi_{-|\xi|^2}(t)\phi_{-\xi}(x)$, which is a solution of the adjoint equation $\partial_t w(x,t)-\Delta w(x,t)=0$. Then by \Cref{LEMMA:Orthogonal relation}, we have
    \begin{equation*}
        F_1^{(r_2)}(F_{r_2}^{1(r_1)}-F_{r_2}^{2(r_1)})\int_0^T H^2_\xi(t)\psi_{-|\xi|^2}(t)dt=0.
    \end{equation*}
    Denote $a=8\pi^2|\xi|^2$, then
    \begin{equation*}
        \begin{aligned}
            \int_0^T H^2_\xi(t)\psi_{-|\xi|^2}(t)dt&=\frac{1}{2}\int_0^T \left[(T-t)e^{-at}+\frac{e^{-aT}-e^{-at}}{a}\right] dt\\
            &=\frac{e^{-aT}(aT+e^{aT}(aT-2)+2)}{a^2},
        \end{aligned}
    \end{equation*}
    therefore $F_1^{(r_2)}\int_0^T H^2_\xi(t)\psi_{-|\xi|^2}(t)dt\neq 0$ when $a$ is large, which implies $F_{r_2}^{1(r_1)}=F_{r_2}^{2(r_1)}$. Therefore we can uniquely determine $F_{r_2}^{(r_1)}$ for all $r_1,r_2\in\{1,2,\dots,n\}$, which includes all of the first order Taylor coefficients for the other populations.

    The higher-order linearizations are similar, and we complete the proof by induction. For the $s$-th ($s\ge 2$) order linearization, we take $f_1^{r_1}=\phi_\xi(x)$, $f_\ell^{r_\ell}=1$, $\ell=2,\dots,s$ for fixed $r_1,\dots,r_s\in\{1,2,\dots,n\}$, and take all other initial values $f_\ell^i=0$.

    The induction assumptions are as follows:

    \begin{enumerate}
        \item \label{ASS:F} All coefficients $F_i^{(\balpha)}$ with $|\balpha|\le s-2$, $i=1,2,\dots,n$ are uniquely determined.

        \item \label{ASS:m} $m_i^{(\balpha)}$ is uniquely determined for all $i=1,2\dots,n$ and $\balpha$ with $|\balpha|\le s-1$ and $\alpha_k \le 1$ for all $k=1,\dots, d$.

        \item \label{ASS:u} $u_i^{(\balpha)}$ is uniquely determined for all $i=1,2\dots,n$ and $\balpha$ with $|\balpha|\le s-2$ and $\alpha_k \le 1$ for all $k=1,\dots, d$.

        \item \label{ASS:Grad u}$\nabla u_i^{(\balpha)}$ is uniquely determined for all $i=1,2\dots,n$ and $\balpha$ with $|\balpha|\le s-1$, $\alpha_k\le 1$ for all $k=1,\dots,d$, and $\alpha_1 = 0$.

        \item \label{ASS:Laplacian u} $\Delta \overline{u}_i^{(1,\dots,s-1)}=\overline{F}_i^{(r_1,\dots,r_{s-1})}H^1_\xi(t)\phi_\xi(x)$. 
    \end{enumerate}

    As usual, we put a bar over a quantity to denote the difference between that quantity for $j=1$ and $j=2$. For example, in assumption \ref{ASS:Laplacian u}, $\overline{u}_i^{(1,\dots,s-1)}$ means $u_i^{1(1,\dots,s-1)} - \overline{u}_i^{2(1,\dots,s-1)}$.
    
    These five assumptions can be seen to hold in the base case $s=2$ from the preceding discussion in the first and second-order linearization steps.
    
    Now assume that the inductive assumptions hold for a given $s\ge 2$. To show that they hold when $s$ is replaced with $s+1$, we consider the $s$-th order linearization of the MFG system \eqref{EQ:MFG MultiPop simpler}. Similar to the derivation of system \eqref{EQ:Second order, x-independent}, we have the following system by the inductive assumptions:
    \begin{equation}\label{EQ:s order, x-independent}
        \left\{\begin{array}{>{\displaystyle}lc}
            -\partial_t \overline{u}_i^{(1,\dots,s)}-\Delta \overline{u}_i^{(1,\dots,s)}\\
            \qquad=\sum_{k_1=1}^n F_i^{(k_1)}\overline{m}_{k_1}^{(1,\dots,s)}+\overline{F}_i^{(r_1,\dots,r_s)}m_{r_1}^{(1)}\cdots m_{r_s}^{(s)}\\
            \qquad\ + \mathbbm{1}_{s > 2} \sum_{k_1=1}^n\cdots\sum_{k_{s-1}=1}^n \overline{F}_i^{(k_1,\dots,k_{s-1})} M_s, &(x,t)\in\bbT^d\times(0,T),\\
            \partial_t \overline{m}_i^{(1,\dots,s)}-\Delta \overline{m}_i^{(1,\dots,s)}=\sum_{\ell=1}^s {\rm div}(m_i^{(\ell)}\nabla \overline{u}_i^{(1,\dots,s,\ell')}) = \sum_{\ell=2}^s \delta_{ir_\ell}\Delta \overline{u}_i^{(1,\dots,s,\ell')}, &(x,t)\in\bbT^d\times(0,T),\\
            \overline{u}_i^{(1,\dots,s)}(x,T)=0,\ \overline{m}_i^{(1,\dots,s)}(x,0)=0, &x\in\bbT^d.
        \end{array}\right.
    \end{equation}
    Here, $M_s$ is a sum of products of $m$-terms, each with order at most $2$ (thus uniquely determined by assumption \ref{ASS:m}). The notation $\mathbbm{1}_{s > 2}$ in front of the last term of the first equation indicates that the term should be omitted in the case $s=2$. We use the notation $(1,\dots,s,\ell')$ to denote $(1,\dots,\ell-1,\ell+1,\dots,s)$, and we will use the notation $(r_1,\dots,r_s,r_\ell')$ in the same way. Note that the $s$-th order linearization of the Hamiltonian term $\frac12 |\nabla u_i|^2$ is uniquely determined on account of assumption \ref{ASS:Grad u} and the fact that $\nabla u_i^{(\ell)} = 0$ for $\ell\neq 1$, so it does not appear in the first equation of \eqref{EQ:s order, x-independent}.

    By assumption \ref{ASS:Laplacian u}, we have
    \begin{equation*}
        \partial_t \overline{m}_i^{(1,\dots,s)}-\Delta \overline{m}_i^{(1,\dots,s)} = \sum_{\ell=2}^s \delta_{ir_\ell}\Delta \overline{u}_i^{(1,\dots,s,\ell')} = \sum_{\ell=2}^s \delta_{ir_\ell}\overline{F}_i^{(r_1,\dots,r_s,r_\ell')}H^1_\xi(t)\phi_\xi(x). 
    \end{equation*}
    Combined with the initial condition $\overline{m}_i^{(1,\dots,s)}(x,0)=0$, the definition of $H_\xi^2(t)$ in the second order linearization step shows that the solution to this initial value problem is 
    \begin{equation}\label{EQ:s order m}
        \overline{m}_i^{(1,\dots,s)} = \sum_{\ell=2}^s \delta_{ir_\ell}\overline{F}_i^{(r_1,\dots,r_s,r_\ell')}H^2_\xi(t)\phi_\xi(x).
    \end{equation}
    Let us plug this into the above equation \eqref{EQ:s order, x-independent} for $\overline{u}_i^{(1,\dots,s)}$. Also, let us set $i=1$ so that the last two terms on the right-hand side vanish (since $F_1(\bm)$ is uniquely determined). We obtain  
    \begin{align*}
        -\partial_t \overline{u}_1^{(1,\dots,s)}-\Delta \overline{u}_1^{(1,\dots,s)} &= \sum_{k_1=1}^n \sum_{\ell=2}^s F_1^{(k_1)}\delta_{k_1 r_\ell}\overline{F}_{k_1}^{(r_1,\dots,r_s,r_\ell')}H^2_\xi(t)\phi_\xi(x)\\
        &= \sum_{\ell=2}^s F_1^{(r_\ell)}\overline{F}_{r_\ell}^{(r_1,\dots,r_s,r_\ell')}H^2_\xi(t)\phi_\xi(x).
    \end{align*}
    
    We also have the terminal and initial conditions $\overline{u}_1^{(1,\dots,s)}(x,T) = \overline{u}_1^{(1,\dots,s)}(x,0)=0$ (the latter of which comes from the data). Then as in the second order linearization step, we can introduce $w(x,t):=\psi_{-|\xi|^2}(t)\phi_{-\xi}(x)$, which is a solution of the adjoint equation $\partial_t w(x,t)-\Delta w(x,t)=0$, and use the same argument as before to get that 
    \begin{equation*}
        0 = \sum_{\ell=2}^s F_1^{(r_\ell)}\overline{F}_{r_\ell}^{(r_1,\dots,r_s,r_\ell')} = \sum_{\ell=1}^s F_1^{(r_\ell)}\overline{F}_{r_\ell}^{(r_1,\dots,r_s,r_\ell')} - F_1^{(r_1)}\overline{F}_{r_1}^{(r_1,\dots,r_s,r_1')}.
    \end{equation*}
    Since $r_1,\dots,r_s$ are arbitrarily taken from $\{1,2,\dots,n\}$, we can replace $r_1,\dots,r_s$ by a cyclic shift $r_k,r_{k+1},\dots,r_{k-1}$ to obtain 
    \begin{equation*}
        0 = \sum_{\ell=1}^s F_1^{(r_\ell)}\overline{F}_{r_\ell}^{(r_1,\dots,r_s,r_\ell')} - F_1^{(r_k)}\overline{F}_{r_k}^{(r_1,\dots,r_s,r_k')},
    \end{equation*}
    for all $k=1,\dots, s$. Upon adding up all these $s$ equations, we arrive at $$(s-1)\sum_{\ell=1}^s F_1^{(r_\ell)}\overline{F}_{r_\ell}^{(r_1,\dots,r_s,r_\ell')} = 0,$$ whence $\sum_{\ell=1}^s F_1^{(r_\ell)}\overline{F}_{r_\ell}^{(r_1,\dots,r_s,r_\ell')} = 0$ and thus $F_1^{(r_k)}\overline{F}_{r_k}^{(r_1,\dots,r_s,r_k')} = 0$, for all $k=1,\dots, s$. However, we are assuming that $F_1^{(r_k)}\neq 0$, so $F_{r_k}^{(r_1,\dots,r_s,r_k')}$ is uniquely determined for all $r_1,\dots,r_s$, which implies that the inductive assumption \ref{ASS:F} holds when $s$ is replaced with $s+1$, i.e. for all $\balpha$ with $|\balpha|\le s-1$.

    It only remains to prove that the remaining inductive assumptions \ref{ASS:m}, \ref{ASS:u}, \ref{ASS:Grad u}, and \ref{ASS:Laplacian u} also hold when $s$ is replaced with $s+1$.

    \paragraph{Assumption \ref{ASS:m}.} We want to show that $m_i^{(\balpha)}$ is uniquely determined for all $i=1,2\dots,n$ and $\balpha$ with $|\balpha| = s$ and $\alpha_k \le 1$ for all $k=1,\dots, d$. We split into cases: $\alpha_1 = 1$ and $\alpha_1 = 0$.

    \subparagraph{Case 1: $\alpha_1=1$.} Let us return to equation \eqref{EQ:s order m} for $\overline{m}_i^{(1,\dots,s)}$. We now know that the right-hand side vanishes since we have just proved that assumption \ref{ASS:F} holds for all $|\balpha|\le s-1$. Therefore, $\overline{m}_i^{(1,\dots,s)} = 0$ and $m_i^{(1,\dots,s)}$ is uniquely determined. By symmetry, the same is true for $m_i^{(\balpha)}$ for all $\balpha$ with $|\balpha|\le s$, $\alpha_k \le 1$ for all $k=1,\dots, d$, and $\alpha_1=1$.

    \subparagraph{Case 2: $\alpha_1 = 0$.} Taking the $s$-th order linearization of the second equation of \eqref{EQ:MFG MultiPop simpler} with respect to $\eps_2,\dots,\eps_{s+1}$ gives 
    \begin{align*}
        \partial_t \overline{m}_i^{(2,\dots,s+1)} - \Delta \overline{m}_i^{(2,\dots,s+1)} = 0,
    \end{align*}
    thanks to assumptions \ref{ASS:m} and \ref{ASS:Grad u}. Together with the initial condition $\overline{m}_i^{(2,\dots,s+1)}(x,0) = 0$, this implies that $\overline{m}_i^{(2,\dots,s+1)} = 0$. That is, $m_i^{(2,\dots,s+1)}$ is uniquely determined. By symmetry, the same is true for $m_i^{(\balpha)}$ for all $\balpha$ with $|\balpha| = s$, $\alpha_k \le 1$ for all $k=1,\dots, d$, and $\alpha_1=0$.

    \paragraph{Assumption \ref{ASS:Laplacian u}.} Let us return to the first equation of \eqref{EQ:s order, x-independent}. Notice that we now know that the last term in the right-hand side vanishes since we have proved that assumption \ref{ASS:F} holds for all $|\balpha|\le s-1$. Furthermore, we know that the first term vanishes since we just proved that $\overline{m}_i^{(1,\dots,s)} = 0$. Lastly, the second term simplifies to just $\overline{F}_i^{(r_1,\dots,r_s)} m_{r_1}^{(1)} = \overline{F}_i^{(r_1,\dots,r_s)} \psi_{-|\xi|^2}(t)\phi_{\xi}(x)$, so the first equation of \eqref{EQ:s order m} ultimately simplifies to 
    \begin{align*}
        -\partial_t \overline{u}_i^{(1,\dots,s)}-\Delta \overline{u}_i^{(1,\dots,s)} = \overline{F}_i^{(r_1,\dots,r_s)} \psi_{-|\xi|^2}(t)\phi_{\xi}(x),
    \end{align*}
    together with the terminal condition $\overline{u}_i^{(1,\dots,s)}(x,T) = 0$. Solving this terminal value problem as in the first-order linearization step and taking the Laplacian of the solution leads to
    \begin{align*}
        \Delta \overline{u}_i^{(1,\dots,s)} = \overline{F}_i^{(r_1,\dots,r_s)}H_{\xi}^1(t)\phi_{\xi}(x).
    \end{align*}
    Thus assumption \ref{ASS:Laplacian u} holds with $s$ replaced with $s+1$, as desired.
    
    \paragraph{Assumption \ref{ASS:u}.}  We want to show that $u_i^{(\balpha)}$ is uniquely determined for all $i=1,2\dots,n$ and $\balpha$ with $|\balpha| = s-1$ and $\alpha_k \le 1$ for all $k=1,\dots, d$. We split into cases: $\alpha_1 = 1$ and $\alpha_1 = 0$. 
    
    \subparagraph{Case 1: $\alpha_1=1$. } Replacing $s$ with $s-1$ in the first equation of \eqref{EQ:s order, x-independent} gives 
    \begin{multline*}
        -\partial_t \overline{u}_i^{(1,\dots,s-1)}-\Delta \overline{u}_i^{(1,\dots,s-1)} =\sum_{k_1=1}^n F_i^{(k_1)}\overline{m}_{k_1}^{(1,\dots,s-1)}\\+\sum_{k_1=1}^n\cdots\sum_{k_{s-1}=1}^n\overline{F}_i^{(k_1,\dots,k_{s-1})}m_{k_1}^{(1)}\cdots m_{k_{s-1}}^{(s-1)} + \mathbbm{1}_{s-1 > 2}\sum_{k_1=1}^n\cdots\sum_{k_{s-2}=1}^n \overline{F}_i^{(k_1,\dots,k_{s-2})} M_{s-1}.
    \end{multline*}
    Observe that all of the quantities on the right-hand side with a bar are, in fact, $0$ due to the inductive assumptions and our previous work. Thus, $\overline{u}_i^{(1,\dots,s-1)}$ solves the backward heat equation 
    \begin{align*}
        -\partial_t \overline{u}_i^{(1,\dots,s-1)}-\Delta \overline{u}_i^{(1,\dots,s-1)} = 0,
    \end{align*}
    with zero terminal condition $\overline{u}_i^{(1,\dots,s-1)}(x,T) = 0$. Hence $\overline{u}_i^{(1,\dots,s-1)} = 0$ and $u_i^{(1,\dots,s-1)}$ is uniquely determined. By symmetry, the same is true for $u_i^{(\balpha)}$ for all $\balpha$ with $|\balpha| = s-1$, $\alpha_k \le 1$ for all $k=1,\dots, d$, and $\alpha_1=1$.

    \subparagraph{Case 2: $\alpha_1=0$.} This is similar to Case 1. Taking the $(s-1)$-th order linearization of the first equation of \eqref{EQ:MFG MultiPop simpler} with respect to $\eps_2,\dots,\eps_s$ gives 
    \begin{multline*}
        -\partial_t \overline{u}_i^{(2,\dots,s)}-\Delta \overline{u}_i^{(2,\dots,s)}
        =\sum_{k_1=1}^n F_i^{(k_1)}\overline{m}_{k_1}^{(2,\dots,s)}\\+\sum_{k_1=1}^n\cdots\sum_{k_{s-1}=1}^n\overline{F}_i^{(k_1,\dots,k_{s-1})}m_{k_1}^{(2)}\cdots m_{k_{s-1}}^{(s)} + \mathbbm{1}_{s-1 > 2}\sum_{k_1=1}^n\cdots\sum_{k_{s-2}=1}^n \overline{F}_i^{(k_1,\dots,k_{s-2})} N_{s-1}.
    \end{multline*}
    Here, $N_{s-1}$ is a sum of products of $m$-terms, each with order at most $2$ (thus uniquely determined by assumption~\ref{ASS:m}). Again, observe that all of the quantities on the right-hand side with a bar are, in fact, $0$ due to the inductive assumptions and our previous work. The rest of the proof proceeds as in Case 1.

    \paragraph{Assumption \ref{ASS:Grad u}.} We want to show that $\nabla u_i^{(\balpha)}$ is uniquely determined for all $i=1,2\dots,n$ and $\balpha$ with $|\balpha| = s$, $\alpha_k\le 1$  for all $k=1,\dots, d$, and $\alpha_1 = 0$. Taking the $s$-th order linearization of the first equation of \eqref{EQ:MFG MultiPop simpler} with respect to $\eps_2,\dots,\eps_{s+1}$ gives 
    \begin{multline*}
        -\partial_t \overline{u}_i^{(2,\dots,s+1)}-\Delta \overline{u}_i^{(2,\dots,s+1)}=\sum_{k_1=1}^n F_i^{(k_1)}\overline{m}_{k_1}^{(2,\dots,s+1)} \\
        +\sum_{k_1=1}^n\cdots\sum_{k_{s}=1}^n\overline{F}_i^{(k_1,\dots,k_{s})}m_{k_1}^{(2)}\cdots m_{k_{s}}^{(s+1)} + \mathbbm{1}_{s > 2}\sum_{k_1=1}^n\cdots\sum_{k_{s-1}=1}^n \overline{F}_i^{(k_1,\dots,k_{s-1})} N_{s}.
    \end{multline*}
    The first and last terms on the right-hand side vanish due to previous work on uniqueness. As for the second term, it simplifies to just the constant $\overline{F}_i^{(r_2,\dots,r_{s+1})}$. Together with the terminal condition $\overline{u}_i^{(2,\dots,s+1)}(x,T) = 0$, we deduce that $\overline{u}_i^{(2,\dots,s+1)}(x,t) = (T-t)\overline{F}_i^{(r_2,\dots,r_{s+1})}$. In particular, $\nabla \overline{u}_i^{(2,\dots,s+1)} = 0$. That is, $\nabla u_i^{(2,\dots,s+1)}$ is uniquely determined. By symmetry, the same is true for $\nabla u_i^{(\balpha)}$ for all $\balpha$ with $|\balpha| = s$, $\alpha_k \le 1$ for all $k=1,\dots, d$, and $\alpha_1=0$.
    
    Then by induction, the assumptions \ref{ASS:F}, \ref{ASS:m}, \ref{ASS:u}, \ref{ASS:Grad u}, \ref{ASS:Laplacian u} hold for all $s\ge 2$. In particular, assumption \ref{ASS:F} holds for all $s\ge 2$. Thus 
    \[
        \bF^1(\bz)=\bF^2(\bz)\ \text{in}\ \bbR^n,
    \]
    and the proof is complete.
\end{proof}

\begin{remark}
    It is important to observe that the condition $F_1^{(k)}\neq 0$ is necessary for the second part of the theorem to hold. This condition ensures different population components are not completely decoupled, as otherwise, we can not expect to be able to reconstruct the cost function for a population that does not interact with the population where the measurement is taken.
\end{remark}

\section{Concluding remarks}
\label{SEC:Concluding}

This work studied inverse problems to a multipopulation mean field game system. Utilizing existing tools in inverse problem theory, such as the multi-linearization technique, we derived results showing that the cost functions can be uniquely reconstructed from multipopulation or single-population data. In the latter case, mild additional assumptions have to be imposed on the form of the cost functions.

Several aspects of our study can be improved. For instance, we studied this MFG system on a torus. Even though this is widely taken when studying MFG models, it would be good to generalize the results to the model in bounded domains, in which case appropriate boundary conditions need to be included. We should also be able to weaken the condition in~\Cref{THM:state-indep recon} on $F_1^{(1)}\neq 0$, as we expect that all we need is to ensure that different populations are not completely decoupled.

One really interesting problem is to see whether or not one could hope to reconstruct the cost functions of all populations from data measured at a single population in a general setup. \Cref{THM:state-indep recon} shows that it is possible to do so in the special case of state-independent cost functions, assuming that different populations are coupled together. Is this still true when the cost functions depend on the state variable? If it is possible to do so, what type of coupling conditions do we need to impose on the coupling of different populations?

\section*{Acknowledgments}

We would like to thank Professor Michael Klibanov (University of North Carolina at Charlotte) for useful suggestions on an early version of the draft. We would like to thank the anonymous referees for their useful comments that helped us improve the quality of this work. This work is partially supported by the National Science Foundation through grants DMS-1937254 and DMS-2309802.
    
{\small

}

\end{document}